\documentclass[11pt]{article}
\usepackage{amsmath, amsthm, amstext}
\usepackage{amssymb, array, amsfonts}
\usepackage[utf8]{inputenc}
\inputencoding{utf8}
\usepackage[english]{babel}
\usepackage{enumerate}
\usepackage{graphicx}

\usepackage[numbers,sort&compress]{natbib}\bibpunct[, ]{[}{]}{,}{n}{,}{,}
\theoremstyle{plain}
\newtheorem{thm}{Theorem}
\newtheorem{lemma}{Lemma}

\newtheorem{cor}{Corollary}

\title{On the equivalence of the BMO-norm of divergence-free vector fields and norm of related paracommutators}
\author{M.~N.~Demchenko\footnote{St.~Petersburg Department of
V.\,A.~Steklov Institute of Mathematics of
the Russian Academy of Sciences, 
27 Fontanka, St.~Petersburg, Russia. E-mail: demchenko@pdmi.ras.ru.\newline
\indent The research was supported by the RFBR grant 20-01-00627-a.}}

\date{}

\begin{document}
\maketitle

\begin{abstract}
We establish an estimate of the BMO-norm of a divergence-free vector field in ${\mathbb R}^3$
in terms of the operator norm of an associated paracommutator. The latter is essentially a $\Psi$DO,
whose symbol depends linearly on the vector field.
Together with the result of P.~Auscher and M.~Taylor concerning the converse estimate,
this provides an equivalent norm in the space of divergence-free fields from BMO.
\smallskip

\noindent \textbf{Keywords:} 
BMO, paracommutators, divergence-free fields. 
\end{abstract}

\section{Introduction}
Let $I$ be a singular integral operator of convolution type bounded in $L_2({\mathbb R}^d)$.
Its compositions $I u$ and $u I$ with a pointwise multiplier $u\in L_\infty({\mathbb R}^d)$, and therefore, the commutator 
\begin{equation}
  I_u = [u, I],
  \label{Ab}
\end{equation}
are bounded in $L_2({\mathbb R}^d)$ as well, and the corresponding operator norm can trivially be estimated in terms of $\|u\|_{L_\infty}$.
As is well known from harmonic analysis,
under some additional assumptions on $I$,
the commutator  $I_u$ can be estimated as follows~\cite{CRW}: 
\begin{equation}
  \|I_u\| \leqslant C \|u\|_{\rm BMO},
  \label{est}
\end{equation}
where $C$ depends only on $d$, $I$
(we recall the definition of the space BMO and some of its properties in sec.~\ref{BMO}).
This fact was generalized to a certain class of paradifferential operators $I_u$ (paracommutators),
which depend on the coefficient $u$ in a more general way, than it is prescribed by~(\ref{Ab}).
In particular, this is true for the operators studied in the present paper, which have the following form.
Consider the Weyl's decomposition for vector-functions in ${\mathbb R}^3$~\cite[Theorems~1.1, 2.1]{BS}:
\[
  L_2({\mathbb R}^3; {\mathbb C}^3) = G \oplus J,
\]
where the subspaces $G$ and $J$ are defined as follows
\begin{gather*}
  G = \overline{G_0}^{L_2}, 
  \quad
  G_0 = \{\partial\varphi\,|\, \varphi\in C_0^\infty({\mathbb R}^3)\},\\
  J = \overline{\{u \in C_0^\infty({\mathbb R}^3; {\mathbb C}^3)\,|\, {\rm div} u = 0\}}^{L_2}.
\end{gather*}
To a vector-function $u\in L_{2,\,\rm loc}({\mathbb R}^3; {\mathbb C}^3)$,
we associate a linear operator $I_u: G_0 \to G$ acting on $f\in G_0$ as follows
\begin{equation}
  I_u f = P (u\times f),
  \label{mapping}
\end{equation}
where $\times$ is the vector product, $P$ is the orthogonal projection on $G$ acting in $L_2({\mathbb R}^3; {\mathbb C}^3)$.
Observe that $u\times f \in L_2({\mathbb R}^3; {\mathbb C}^3)$, which means that $I_u f$ is well-defined.
The operator $I_u$ assumes a continuation to a bounded operator in $G$,
provided that the norm $\|I_u\|$ is finite, the latter being defined by
\[
  \|I_u\| = \sup_{f\in G_0\setminus\{0\}} \frac{\|I_u f\|_{L_2}}{\|f\|_{L_2}}.
\]
(Note also that if $\|I_u\|<\infty$ and $u$ is a real-valued vector-function, then
$I_u$ is a bounded skew-symmetric operator, i.e. $I_u^* = -I_u$.)
The operator $I_u$ satisfies estimate~(\ref{est}), if $u\in{\rm BMO}$~\cite{AT}, \cite[chap.~3, sec.~8]{Taylor}. 

The present paper concerns the converse of estimate~(\ref{est}) for the operators of the form~(\ref{mapping}): 
\begin{equation}
  \|u\|_{\rm BMO} \leqslant C \|I_u\|.
  \label{west}
\end{equation}
Such estimates are known 
in the case when a linear mapping $u \mapsto I_u$ 
satisfies a certain nondegeneracy condition (see~\cite{Janson Peetre} and the literature cited therein).
In particular, for the commutator~(\ref{Ab}), this result was established in~\cite{Janson}.
To our knowledge, none of the previous results applies when $I_u$ is defined by relation~(\ref{mapping}).
In fact, such a relation of the field $u$ and the operator $I_u$ is degenerate, which
is demonstrated by
\begin{thm} \label{thm0}
  Let $u\in L_{2,\,\rm loc}({\mathbb R}^3; {\mathbb C}^3)$ satisfy the relation ${\rm curl}\, u = 0$. Then $I_u = 0$.
\end{thm}
This simple (as it will be seen from the proof) fact indicates that estimate~(\ref{west}) can be valid only
under some additional assumptions on $u$.
Sufficient conditions are given by
\begin{thm}\label{thm}
  Let $u\in L_{2,\,\rm loc}({\mathbb R}^3; {\mathbb C}^3)$ satisfy the condition
  \begin{equation}
    \int_{{\mathbb R}^3} \frac{|u(x)|}{1+|x|^4}\, dx < \infty,
    \label{converge-int}
  \end{equation}
  the relation ${\rm div} u = 0$, and $\|I_u\|<\infty$.
  Then $u$ belongs to BMO and satisfies estimate~(\ref{west}).
\end{thm}

It should be stressed that the condition $u\in L_{2,\,\rm loc}({\mathbb R}^3; {\mathbb C}^3)$ concerning local behavior of $u$
and condition~(\ref{converge-int}) concerning behavior of $u$ at the infinity,
which both occur in Theorem~\ref{thm}, 
are not too restrictive in the sense that
they are satisfied by any $u$ from BMO.
Thus in view of estimate~(\ref{est}) for the operators of the form~(\ref{mapping}) and Theorem~\ref{thm}, we obtain
\begin{cor}\label{cor}
  There are positive constants $c$, $C$, such that
  for any $u$ from BMO satisfying ${\rm div} u = 0$, we have
  \[
    c \|I_u\| \leqslant \|u\|_{\rm BMO} \leqslant C \|I_u\|.
  \]
\end{cor}

\section{Proof of Theorem~\ref{thm0}}
It suffices to show that under the assumptions of Theorem~\ref{thm0},
the field $u\times f$ is orthogonal to $G$ for any $f\in G_0$.
For $u\in C^1({\mathbb R}^3; {\mathbb C}^3)$, this follows from the fact that for an arbitrary function $\varphi\in C_0^\infty({\mathbb R}^3)$, we have
\[
  \int_{{\mathbb R}^3} \langle{}u\times f, \overline{\partial\varphi}\rangle dx = -\int_{{\mathbb R}^3} {\rm div}(u\times f) \overline\varphi\, dx = 0
\]
(here and further $\langle{}u,v\rangle$ is the bilinear form in ${\mathbb C}^3$ defined on the standard orthonormal frame 
$\{e_\alpha\}_{\alpha=1,2,3}$ by the equality $\langle{}e_\alpha, e_\beta\rangle = \delta_{\alpha\beta}$).
In the last equality we used the identity
\begin{equation}
  {\rm div}(u\times f) = \langle{\rm curl}\, u, f\rangle - \langle{}u, {\rm curl}\, f\rangle
  \label{divcross}
\end{equation}
and the relation ${\rm curl}\, f = 0$.

In the case of nonsmooth $u$, one can apply a smoothing mollifier, which provides smooth functions $u^\varepsilon$, $\varepsilon>0$, 
${\rm curl}\, u^\varepsilon=0$,
approximating $u$ in the $L_1$-norm on the intersection of the supports of $f$ and $\varphi$.
We have
\begin{multline*}
  \int_{{\mathbb R}^3} \langle{}u\times f, \overline{\partial\varphi}\rangle dx = \int_{{\mathbb R}^3} \langle{}u, f\times\overline{\partial\varphi}\rangle dx =
  \lim_{\varepsilon\to 0} \int_{{\mathbb R}^3} \langle{}u^\varepsilon, f\times\overline{\partial\varphi}\rangle dx \\
  =\lim_{\varepsilon\to 0} \int_{{\mathbb R}^3} \langle{}u^\varepsilon\times f, \overline{\partial\varphi}\rangle dx = 0.
\end{multline*}

\section{BMO space} 
\label{BMO}
A function $u\in L_{1,\,\rm loc}({\mathbb R}^d; {\mathbb C}^N)$ is said to belong to BMO if
\begin{equation}
  \|u\|_{\rm BMO} = \sup_Q |Q|^{-1}\int_Q |u(x)-u_Q|\, dx < \infty,
  \label{BMOnorm}
\end{equation}
where the supremum is taken over all cubes $Q\subset{\mathbb R}^d$,
\[
  u_Q = |Q|^{-1}\int_Q u(x)\, dx.
\]
For constant functions $u=\rm const$ (and only for them), we have $\|u\|_{\rm BMO} = 0$.
Thus the functional $\|u\|_{\rm BMO}$ is a {\em seminorm}, rather than a norm.
However, we will call it the BMO-norm. 

Functions from BMO belong to $L_{p,\,\rm loc}({\mathbb R}^d)$ for any $1\leqslant p < \infty$, 
and the norm~(\ref{BMOnorm}) is equivalent to
\[
  \sup_Q \left(|Q|^{-1}\int_Q |u(x)-u_Q|^p\, dx\right)^{1/p}.
\]
Recall that we have already used the fact that ${\rm BMO} \subset L_{2,\,\rm loc}({\mathbb R}^d)$ when we obtained Corollary~\ref{cor}.

We will need certain characterizations of $L_2$ and BMO spaces, which differ from their standard definitions.
Let a real-valued radially symmetric function $\Phi$ from $C_0^\infty({\mathbb R}^d)$ 
satisfy the following conditions
\begin{equation}
  \Phi \not\equiv 0, \quad \widehat\Phi(0) = 0
  \label{condPhi}
\end{equation}
($\hat f$ denotes Fourier transform of $f$).
In this case, we have
\begin{equation}
  \int_0^\infty \widehat\Phi(t \xi)^2 \frac{dt}{t} = c^2>0, \quad \xi\ne 0
  \label{S}
\end{equation}
(due to~(\ref{condPhi}), the integral converges and $c$ is positive).
Then $L_2$ space can be characterized in terms of convolutions of its functions with the functions
\begin{equation}
  \Phi_t(x) = t^{-d} \Phi(t^{-1} x), \quad t > 0.
  \label{Phik}
\end{equation}
Namely, for any function $u\in L_2({\mathbb R}^d)$, we have
\begin{equation}
  \left(\int_0^\infty \|\Phi_t * u\|_{L_2}^2 \frac{dt}{t}\right)^{1/2} = c \|u\|_{L_2},
  \label{PhikL2}
\end{equation}
where $c$ is the same constant as in equality~(\ref{S}).
This fact is a trivial consequence of Plancherel's theorem,
since $\widehat\Phi_t(\xi) = \widehat \Phi(t \xi)$.
Thus the convolutions $\Phi_t * u$ provide an equivalent norm in $L_2$ space.
As is known, various function spaces are characterized in terms of such convolutions~\cite{Stein, Triebel, Grafakos},
which is far less trivial than the same fact for $L_2$ space.
This concerns Sobolev spaces (or, more generally, Triebel-Lizorkin spaces),
Besov's spaces, Hardy space, and BMO space.
In most cases, conditions similar to~(\ref{condPhi}) (more strong ones, as a rule) are imposed on $\Phi$.
Now we give a characterization of BMO space, which will be used later.
For any $u$ from BMO, we have~\cite{Stein}
\[
  \sup_Q \left(\int_0^{l(Q)} \|\Phi_t * u\|_Q^2\, \frac{dt}{t}\right)^{1/2} \leqslant C \|u\|_{\rm BMO},
\]
where $l(Q)$ is the side length of the cube $Q$, 
and the functional $\|\cdot\|_Q$ is defined as follows
\[
  \|u\|_Q = \left(|Q|^{-1} \int_Q |u|^2 dx\right)^{1/2}.
\]
The converse estimate holds as well:
if a function $u\in L_{1,\,\rm loc}({\mathbb R}^d; {\mathbb C}^N)$ satisfies condition
\begin{equation}
  \int_{{\mathbb R}^d} \frac{|u(x)|}{1+|x|^{d+1}}\, dx < \infty,
  \label{converge-int-d}
\end{equation}
then~\cite{Stein}
\begin{equation}
  \|u\|_{\rm BMO} \leqslant C \sup_Q \left(\int_0^{l(Q)} \|\Phi_t * u\|_Q^2\, \frac{dt}{t}\right)^{1/2}.
  \label{BMOCarleson}
\end{equation}
Condition~(\ref{converge-int}) in Theorem~\ref{thm} is nothing else than condition~(\ref{converge-int-d}) in the case $d=3$.
For any function from BMO, condition~(\ref{converge-int-d}) holds automatically.
Thus the supremum occurring in the last two estimates yields an equivalent norm in BMO space.

\section{Auxiliary assertions} 
In the proof of Theorem~\ref{thm}, estimate~(\ref{BMOCarleson}) will be used.
However, 
$\Phi$ will be replaced by the function $\Psi = \Phi * \Phi$ with $\Phi = \Delta\Theta$, 
where 
\begin{equation}
  \Theta \in C_0^\infty(\{x\in{\mathbb R}^3\,|\,\, |x|<1/2\})
  \label{PhiTe}
\end{equation}
is a real-valued radially symmetric function, $\Theta\not\equiv 0$. 
Clearly, $\Phi, \Psi \in C_0^\infty({\mathbb R}^3)$. 
Besides, the functions $\Phi$ and $\Psi$ satisfy conditions~(\ref{condPhi}). 

For $\alpha=1,2,3$, introduce the following vector-functions
\[
  \Phi^{\alpha} = \partial \partial_\alpha \Theta. 
\]
Let $\Phi_t$, $\Phi^\alpha_t$, $\Psi_t$ ($t>0$) be the families of functions related to $\Phi$, $\Phi^\alpha$, $\Psi$,
respectively, by equality of the form~(\ref{Phik}) with $d=3$.

To a vector-function $u$, we associate a family of vector-functions
\[
  u^\alpha_t = \Phi^\alpha_t \,\vec{*}\, u, \quad \alpha=1,2,3, \quad t>0,
\]
where the ``vector convolution'' operation $\vec{*}$ is defined by
\[
  (u \,\vec{*}\, v)(x) = \int_{{\mathbb R}^3} u(x-y) \times v(y)\, dy.
\]

In the proof of Theorem~\ref{thm}, we will estimate the convolutions $\Psi_t * u$.
To this end, we will need the following lemma, which relates them to the functions $u^\alpha_t$.
\begin{lemma}
For $u\in L_{1,\,\rm loc}({\mathbb R}^3; {\mathbb C}^3)$, ${\rm div} u = 0$, $t>0$, we have
\begin{equation}
  \Psi_t * u = -\sum_{\alpha=1,2,3} \Phi^\alpha_t \,\vec{*}\, u^\alpha_t.
  \label{PsiPhi}
\end{equation}
\end{lemma}
\begin{proof}
Since $\Psi_t$, $\Phi^\alpha_t$ are compactly supported functions,
both sides of equa\-lity~(\ref{PsiPhi}) are smooth functions in ${\mathbb R}^3$.
It is possible to approximate $u$ by smooth divergence-free convolutions with an appropriate smooth compactly supported function.
So it will suffice to establish~(\ref{PsiPhi}) for smooth functions $u$.

Now fix an arbitrary point $x\in{\mathbb R}^3$ and choose a closed ball $B$ centered at $x$
(its radius will be specified later).
In the ball $B$, we have
\[
  u = {\rm curl}\, z, \quad z\in C^\infty(B; {\mathbb C}^3),
\]
which is a consequence of the condition ${\rm div} u = 0$.
Choose an arbitrary smooth compactly supported extension of $z$ to ${\mathbb R}^3$, which will be denoted by the same symbol,
and define the field $v = {\rm curl}\, z$ entirely in ${\mathbb R}^3$.
Since $u=v$ in $B$, we will have
\[
  (\Psi_t * u)(x) = (\Psi_t * v)(x), \quad
  \left(\Phi^\alpha_t \,\vec{*}\, u^\alpha_t\right)(x) = \left(\Phi^\alpha_t \,\vec{*}\, v^\alpha_t\right)(x),
\]
as soon as we choose the radius of $B$ sufficiently large
that relations $\Psi_t(x-y) = 0 = \Phi^\alpha_t(x-y)$ hold for any $y\in {\mathbb R}^3\setminus B$.
Therefore we will obtain~(\ref{PsiPhi}) at $x$,
as soon as we establish the same relation for $v$ (incidentally, the latter will be established entirely in ${\mathbb R}^3$
rather than at the single point $x$). 
We have
\[
  \widehat \Phi^\alpha(\xi) = \widehat \Phi(\xi)\, \langle\xi, e_\alpha\rangle \frac{\xi}{|\xi|^2},
  \quad
  \widehat \Phi^\alpha_t(\xi) = \widehat \Phi^\alpha(t\xi) = \widehat \Phi_t(\xi)\, \langle\xi, e_\alpha\rangle \frac{\xi}{|\xi|^2},
\]
whence
\[
  \hat v^\alpha_t(\xi) = \widehat \Phi^\alpha_t(\xi) \times \hat v(\xi) =
  \widehat \Phi_t(\xi)\, \langle\xi, e_\alpha\rangle\, \frac{\xi\times \hat v(\xi)}{|\xi|^2}. 
\]
Therefore
\begin{multline*}
  \left(\sum_{\alpha=1,2,3} \Phi^\alpha_t \,\vec{*}\, v^\alpha_t\right)\widehat{}\,(\xi) =
  \sum_{\alpha=1,2,3} \widehat\Phi^\alpha_t(\xi)\times \hat v^\alpha_t(\xi) =
  \sum_{\alpha=1,2,3} \widehat\Phi_t(\xi)^2 \langle\xi, e_\alpha\rangle^2\, \frac{\xi\times(\xi\times \hat v(\xi))}{|\xi|^4} \\
  =\widehat\Phi_t(\xi)^2\, \frac{\xi\times(\xi\times \hat v(\xi))}{|\xi|^2} =
  -\widehat\Phi_t(\xi)^2 \hat v(\xi) = -\widehat\Psi_t(\xi) \hat v(\xi)
\end{multline*}
(we used the equality $\langle\hat v(\xi), \xi\rangle = 0$, which follows from ${\rm div} v = 0$).
Thus we arrive at relation~(\ref{PsiPhi}) for the function $v$.
\end{proof}

\begin{lemma}
For $u\in C^1({\mathbb R}^3; {\mathbb C}^3)$, we have
\begin{equation}
  u^\alpha_t = t (\partial_\alpha\Theta)_t * {\rm curl}\, u,   
  \quad  
  \alpha=1,2,3, \quad t>0,
  \label{wkrot}
\end{equation}
where $(\partial_\alpha\Theta)_t$ is related to $\partial_\alpha\Theta$ by equality of the form~(\ref{Phik}) with $d=3$.
\end{lemma}
\begin{proof}
The left hand side of~(\ref{wkrot}) equals
\[
  \Phi^\alpha_t \,\vec{*}\, u = t (\partial (\partial_\alpha\Theta)_t) \,\vec{*}\, u. 
\]
So the desired assertion follows immediately from the identity
\[
  (\partial f) * g = f * \partial g
\]
for scalar functions, in which $f$ should be replaced by $t (\partial_\alpha\Theta)_t$,
and $g$ should be replaced by the components of the vector-function $u$.
\end{proof}

In the proof of the following lemma, we will use the representaion of the orthogonal projection $P$
in terms of Fourier transform ($v\in L_2({\mathbb R}^3; {\mathbb C}^3)$):
\begin{equation}
  \widehat{P v}(\xi) = \frac{\xi}{|\xi|^2}\, \langle\hat v(\xi), \xi\rangle.
  \label{PFourier}
\end{equation}
\begin{lemma}
For $u\in C^1({\mathbb R}^3; {\mathbb C}^3)$, $f\in G_0$, the following equality holds true
\begin{equation}
  \Phi_t * I_u f = t (\partial\Theta)_t * \langle{\rm curl}\, u, f\rangle,   
  \quad  
  t>0.
  \label{Phikrot}
\end{equation}
\end{lemma}
\begin{proof}
The scalar function $\langle{\rm curl}\, u, f\rangle$ belongs to $L_2({\mathbb R}^3)$, 
so equality~(\ref{Phikrot}) is equivalent to the following equality of functions from $L_2({\mathbb R}^3; {\mathbb C}^3)$ of the variable $\xi$
\begin{equation}
  -t^2 |\xi|^2\widehat\Theta(t \xi) (P(u\times f))\,\hat{}\,(\xi) = 
  t^2 i \xi\, \widehat\Theta(t \xi) (\langle{\rm curl}\, u, f\rangle)\,\hat{}\,(\xi).
  \label{Phikrotprime}
\end{equation}
In view of~(\ref{PFourier}), we have
\[
  |\xi|^2 (P(u\times f))\,\hat{}\,(\xi) = \xi\, \langle(u\times f)\,\hat{}\,(\xi), \xi\rangle.
\]
However, the inner product $\langle(u\times f)\,\hat{}\,(\xi), \xi\rangle$ 
equals Fourier transform of $-i\, {\rm div}(u\times f)$, the latter being an element of $L_2({\mathbb R}^3)$, hence
\[
  \langle(u\times f)\,\hat{}\,(\xi), \xi\rangle = -i ({\rm div}(u\times f))\,\hat{}\,(\xi) = -i (\langle{\rm curl}\, u, f\rangle)\,\hat{}\,(\xi).
\]
In the second equality we used identity~(\ref{divcross}).
Thus we obtain equality~(\ref{Phikrotprime}), and so equality~(\ref{Phikrot}) as well.
\end{proof}

\section{Proof of Theorem~\ref{thm} in the case of smooth $u$}\label{smoothproof} 
In this section, we prove Theorem~\ref{thm} under the additional assumption $u\in C^1({\mathbb R}^3; {\mathbb C}^3)$. 

Let $Q$ be an arbitrary cube in ${\mathbb R}^3$. Fix $\beta \in \{1,2,3\}$ and choose a function $\tilde f\in G_0$ that
is equal to $e_\beta$ on the cube centered at the origin 
with side length $1+\sqrt 3$.
Then for the function
\[
  f(x) = \tilde f\left(\frac{x - c(Q)}{l(Q)}\right)
\]
($c(Q)$ is the center of the cube $Q$), we have
\begin{gather}
  f\big|_{Q'} = e_\beta,  \label{feb}\\
  \|I_u f\|_{L_2} \leqslant \|I_u\| \cdot \|f\|_{L_2} \leqslant C \|I_u\|\cdot |Q|^{1/2},
  \label{Iwf}
\end{gather}
where $Q'$ is the cube centered at $c(Q)$ with side length $(1+\sqrt 3)\, l(Q)$.
Due to~(\ref{Phikrot}), for $\alpha=1,2,3$, $t>0$, we have
\begin{equation}
  \langle\Phi_t * I_u f, e_\alpha\rangle = 
  t (\partial_\alpha\Theta)_t * \langle{\rm curl}\, u, f\rangle. 
  \label{PhikIw}
\end{equation}
In the case $t<l(Q)$,
the convolution occurring on the right hand side being restricted on the cube $Q$ is determined by
the values of the function $\langle{\rm curl}\, u, f\rangle$ on $Q'$,
which follows from~(\ref{PhiTe}).
This can be recorded as follows
\[
  \chi_Q (t (\partial_\alpha\Theta)_t * \langle{\rm curl}\, u, f\rangle) = 
  \chi_Q (t (\partial_\alpha\Theta)_t * (\chi_{Q'} \langle{\rm curl}\, u, f\rangle))
\]
(here and further $\chi_E$ is the characteristic function of a set $E$).
According to~(\ref{feb}), on $Q'$ we have $\langle{\rm curl}\, u, f\rangle = \langle{\rm curl}\, u, e_\beta\rangle$.
Hence the expression we have obtained equals
\[
  \chi_Q \left(t (\partial_\alpha\Theta)_t * (\chi_{Q'} \langle{\rm curl}\, u, e_\beta\rangle)\right) = 
  \chi_Q \left(t (\partial_\alpha\Theta)_t * \langle{\rm curl}\, u, e_\beta\rangle\right). 
\]
Thus the convolution~(\ref{PhikIw}) on the cube $Q$ equals
\[
  t (\partial_\alpha\Theta)_t * \langle{\rm curl}\, u, e_\beta\rangle = \langle{}t (\partial_\alpha\Theta)_t * {\rm curl}\, u, e_\beta\rangle = \langle{}u^\alpha_t, e_\beta\rangle
\]
(the last equality follows from~(\ref{wkrot})).
Therefore
\[
  \|\langle{}u^\alpha_t, e_\beta\rangle\|_Q = \|\langle\Phi_t * I_u f, e_\alpha\rangle\|_Q \leqslant |Q|^{-1/2} \|\langle\Phi_t * I_u f, e_\alpha\rangle\|_{L_2}
  \leqslant |Q|^{-1/2} \|\Phi_t * I_u f\|_{L_2}.
\]
Hence
\begin{multline*}
  \left(\int_0^{l(Q)} \|\langle{}u^\alpha_t, e_\beta\rangle\|_Q^2\, \frac{dt}{t} \right)^{1/2} \leqslant 
  C |Q|^{-1/2} \left(\int_0^{l(Q)} \|\Phi_t * I_u f\|_{L_2}^2\, \frac{dt}{t}\right)^{1/2} \\
  \leqslant C |Q|^{-1/2} \left(\int_0^\infty \|\Phi_t * I_u f\|_{L_2}^2\, \frac{dt}{t} \right)^{1/2} 
  \leqslant
  C |Q|^{-1/2} \|I_u f\|_{L_2} \leqslant 
  C \|I_u\|.
\end{multline*}
In the second to last inequality we applied~(\ref{PhikL2}),
whereas in the last one we used~(\ref{Iwf}).
Now varying $\beta=1,2,3$, yields
\begin{equation}
  \left(\int_0^{l(Q)} \|u^\alpha_t\|_Q^2\, \frac{dt}{t}\right)^{1/2} \leqslant C \|I_u\|.
  \label{wIw}
\end{equation}

Now we are able to estimate the convolutions $\Psi_t * u$.
According to~(\ref{PsiPhi}), we have
\[
  \|\Psi_t * u\|_Q \leqslant \sum_{\alpha=1,2,3} \|\Phi^\alpha_t \,\vec{*}\, u^\alpha_t\|_Q =
  |Q|^{-1/2} \sum_{\alpha=1,2,3} \| \chi_Q (\Phi^\alpha_t \,\vec{*}\, u^\alpha_t)\|_{L_2}.
\]
For $t<l(Q)$, we may apply the equality
\[
  \chi_Q (\Phi^\alpha_t \,\vec{*}\, u^\alpha_t) = \chi_Q (\Phi^\alpha_t \,\vec{*}\, (\chi_{Q'} u^\alpha_t))
\]
(we have already used a similar argument previously),
and proceed with our calculation
\begin{multline*}
  |Q|^{-1/2} \sum_{\alpha=1,2,3} \| \chi_Q (\Phi^\alpha_t \,\vec{*}\, (\chi_{Q'} u^\alpha_t))\|_{L_2} \leqslant
  |Q|^{-1/2} \sum_{\alpha=1,2,3} \|\Phi^\alpha_t \,\vec{*}\, (\chi_{Q'} u^\alpha_t)\|_{L_2} \\
  \leqslant
  C |Q|^{-1/2} \sum_{\alpha=1,2,3} \|\chi_{Q'} u^\alpha_t\|_{L_2} \leqslant C \sum_{\alpha=1,2,3} \|u^\alpha_t\|_{Q'}.
\end{multline*}
Whence
\[
  \left(\int_0^{l(Q)} \|\Psi_t * u\|_Q^2\, \frac{dt}{t} \right)^{1/2} \leqslant
  C \sum_{\alpha=1,2,3} \left(\int_0^{l(Q)} \|u^\alpha_t\|_{Q'}^2\, \frac{dt}{t}\right)^{1/2} \leqslant C \|I_u\|
\]
(the last inequality follows from estimate~(\ref{wIw}) applied to the cube $Q'$).

\section{Proof of Theorem~\ref{thm} in the general case} 
Under the assumptions of Theorem~\ref{thm}, we may turn from $u$ to its smooth approximations $u^\varepsilon = \omega^\varepsilon * u$, $\varepsilon>0$,
where 
\begin{gather*}
  \omega\in C_0^\infty({\mathbb R}^3), \quad
  \omega \geqslant 0, \quad
  \int_{{\mathbb R}^3} \omega(x)\, dx = 1, \quad
  \omega^\varepsilon(x) = \varepsilon^{-3}\omega(\varepsilon^{-1} x).
\end{gather*}
Clearly, the field $u^\varepsilon$ belongs to $C^\infty({\mathbb R}^3; {\mathbb C}^3)$, satisfies condition~(\ref{converge-int})
and the relation ${\rm div} u^\varepsilon = 0$.
Besides, we have $\|I_{u^\varepsilon}\| \leqslant \|I_u\|$.
Indeed, the operator $I_{u^\varepsilon}$ can be represented as follows
\[
  I_{u^\varepsilon} f = \int_{{\mathbb R}^3} \omega^\varepsilon(y) I_{u(\cdot - y)} f dy.
\]
Evidently, $\|I_{u(\cdot - y)}\| = \|I_u\|$, so by applying Minkowski's inequality, we obtain
\begin{multline*}
  \|I_{u^\varepsilon} f\| = \bigg\|\int_{{\mathbb R}^3} \omega^\varepsilon(y) I_{u(\cdot - y)} f dy\bigg\| \leqslant 
  \int_{{\mathbb R}^3} \omega^\varepsilon(y) \|I_{u(\cdot - y)} f\| dy \leqslant
  \int_{{\mathbb R}^3} \omega^\varepsilon(y) \|I_{u(\cdot - y)}\|\, \|f\| dy \\
  = \|I_u\|\, \|f\| \int_{{\mathbb R}^3} \omega^\varepsilon(y) dy = \|I_u\|\, \|f\|.
\end{multline*}
As was shown in sec.~\ref{smoothproof}, we have
\[
  \|u^\varepsilon\|_{\rm BMO} \leqslant C \|I_{u^\varepsilon}\| \leqslant C \|I_u\|.
\]
We will establish the same fact for $u$ with the use of the standard definition of BMO-norm~(\ref{BMOnorm}).
We have $(u^\varepsilon)_Q \to u_Q$, $\varepsilon\to 0$, for any cube $Q$.
Hence by applying Fatou's lemma we obtain
\[
  |Q|^{-1} \int_Q |u(x) - u_Q| dx \leqslant \sup_{\varepsilon > 0} |Q|^{-1} \int_Q |u^\varepsilon(x) - (u^\varepsilon)_Q| dx \leqslant 
  \sup_{\varepsilon > 0} \|u^\varepsilon\|_{\rm BMO} \leqslant C \|I_u\|,
\]
which implies the desired assertion.

\end{document}